\documentclass[12pt,fleqn]{article}
\usepackage{authblk}
\usepackage{orcidlink}
\usepackage[misc]{ifsym}

\usepackage[utf8]{inputenc}
\usepackage[english]{babel}

\usepackage{amsthm}
\usepackage{amssymb}
\usepackage{amsmath}
\usepackage{paralist}
\usepackage{hyperref}

\def\IN{{\mathbb N}}

\def\dom{\mathop{\rm dom}}

\newtheorem{defi}{Definition}
\newtheorem{satz}[defi]{Theorem}
\newtheorem{lem}[defi]{Lemma}
\newtheorem{prop}[defi]{Proposition}
\newtheorem{kor}[defi]{Corollary}

\begin{document}

\title{Reordered Computable Numbers}
\author{Philip Janicki
	 \href{mailto:philip.janicki@unibw.de}{\Letter}
	 \orcidlink{0009-0008-9063-028X}}
\affil{Fakultät für Informatik \\
	 Universität der Bundeswehr München \\
}
\date{\today}
\maketitle

\begin{abstract}
	A real number is called \emph{left-computable} if there exists a computable increasing sequence of rational numbers converging to it. In this article we are investigating a proper subset of the left-computable numbers. We say that a real number $x$ is \emph{reordered computable} if there exist a computable function $f \colon \mathbb{N} \to \mathbb{N}$ with $\sum_{k=0}^{\infty} 2^{-f(k)} = x$ and a bijective function $\sigma \colon \mathbb{N} \to \mathbb{N}$ such that the rearranged series $\sum_{k=0}^{\infty} 2^{-f(\sigma(k))}$ converges computably. In this article we will give some examples and counterexamples for reordered computable numbers and we will show that these numbers are closed under addition, multiplication and the Solovay reduction. 
	Finally, we will also present a density theorem for reordered computable numbers.
\end{abstract}


\section{Introduction}\label{sec:introduction}

In contrast to classical computability theory, computable analysis focuses on the computability of real-valued functions and the effectiveness of real numbers. Usually real numbers are not given directly but they are described by sequences of rational numbers converging to it. There are three important subsets of the real numbers, which are also mentioned in this article. 
A real number is called \emph{computably approximable} if there exists a computable sequence of rational numbers converging to it. 
A real number is called \emph{left-computable} if there exists a computable increasing sequence of rational numbers converging to it. 
A real number is called \emph{computable} if there exists a computable sequence of rational numbers converging computably to it. 
Every computable number is left-computable and every left-computable number is computably approximable, but none of these implications can be reversed.

From a computability-theoric point of view, the computable numbers are certainly the most important subset of the real numbers. However, the left-computable numbers play a very important role in algorithmic information theory. In this area, the real numbers are often represented by dyadic series. In this article we are interested in a specific subset of the left-computable numbers, which we will describe and define by computable dyadic series converging even computably after they have been rearranged. Most of the results of this paper have been presented \cite{Jan2023} at the conference CCA 2023.

This article consists of seven sections. After the introduction in Section~\ref{sec:introduction} and some preliminaries in Section \ref{sec:preliminaries}, we will motivate and define the \emph{reordered computable numbers} in Section~\ref{sec:reordered-computable-numbers}. We will also give a characterization and show that every regular number is reordered computable. For the theoretical background on regular numbers we refer to Wu~\cite{Wu05a}. After this, Section~\ref{sec:calculating-with-limits} contains a collection of several technical lemmas which will be used in the following sections. In Section~\ref{sec:closure-properties} we will show that the reordered computable numbers are closed under addition and multiplication and that they are closed downwards under the Solovay reduction. 
In Section~\ref{sec:counterexamples} we will show some counterexamples for reordered computable numbers. We will prove that reordered computable numbers which are nearly computable are even computable. This result also implies that reordered computable numbers are not Martin-Löf random. For the theoretical background on nearly computable numbers we refer to Hertling and Janicki~\cite{HJ2023}.
In Section~\ref{sec:characteristic-dyadic-values} we will define the \emph{characteristic dyadic value} for left-computable numbers, which is a real number in the interval $\left[1, 2\right]$ and describes how efficiently this number can be approximated by some computable dyadic series. We will show that the characteristic dyadic value of left-computable numbers which are not reordered computable can be only $2$. And we show that for every computably approximable number in $\left[1,2\right]$ there exists a reordered computable number which is not regular with this characteristic dyadic value. This last main result is proved by an infinite diagonalization.

\section{Preliminaries}\label{sec:preliminaries}

Let $(x_n)_n$ be a sequence of real numbers. We say that $(x_n)_n$ is \emph{increasing} if we have $x_{n} < x_{n+1}$ for all $n \in \IN$. If we have $x_{n} \leq x_{n+1}$ for all $n \in \IN$, we say that $(x_n)_n$ is \emph{non-decreasing}. Analogously, we define a \emph{decreasing} and a \emph{non-increasing} sequence.

Let $(x_n)_n$ be a convergent sequence of real numbers and $x := \lim_{n\to\infty}  x_n$. A function $f\colon \IN \to \IN$ is called a \emph{modulus of convergence} for $(x_n)_n$ if for all~${n \in \IN}$ and for all $k \geq f(n)$ we have $\left|x - x_k\right| \leq 2^{-n}$. We say that $(x_n)_n$ \emph{converges computably} to $x$ if it has a computable modulus of convergence.
For any sequence of rational numbers $(a_n)_n$, the series $\sum_{k=0}^{\infty} a_k$ is called \emph{computable} if the sequence of the partial sums $\left(\sum_{k=0}^{n-1} a_k\right)_n$ is computable. Analogously, $\sum_{k=0}^{\infty} a_k$ \emph{converges computably} if $\left(\sum_{k=0}^{n-1} a_k\right)_n$ converges computably. A series of the form $\sum_{k=0}^{\infty} 2^{-f(k)}$ for some function $f\colon \IN \to \IN$ is called \emph{dyadic}.

A function $f \colon \IN \to \IN$ \emph{tends to infinity} if for every constant $c \in \IN$ there exists a number $m \in \IN$ such that for all $n \geq m$ we have $f(n) > c$. In particular, a function tending to infinity is not bounded.

For better readability we will write $n/2$ instead of $\lfloor \frac{n}{2} \rfloor$ for any $n \in \IN$.
As usual, we will denote with $\left|S\right|$ the number of elements of any finite set~$S$. For any set $A \subseteq \IN$, we define $x_A := \sum_{n \in A} 2^{-(n+1)}$, which is a real number in the interval $\left[0, 1\right]$. It is well-known that $A$ is computable if and only if $x_A$ is computable. If $A$ is only assumed to be computably enumerable, then $x_A$ is a left-computable number; the converse is not true as pointed out by Jockusch (see Soare~\cite{Soa69a}). We say that a real number $x \in \left[0,1\right]$ is \emph{strongly left-computable} if there exists a computably enumerable set $A \subseteq \IN$ with $x_A = x$.

We will also use the Cantor pairing function $\left\langle \cdot, \cdot \right\rangle \colon \IN^2 \to \IN$, which is defined by $\left\langle i,j\right\rangle := \frac{1}{2} \cdot (i+j) \cdot (i+j+1) + j$ for all $(i,j) \in \IN^2$. It is well-known that this function is computable and bijective.

Finally, we fix some standard enumeration $(\varphi_e)_e$ of all computable partial functions from $\IN$ to $\IN$. As usual, the notation $\varphi_{e}(n)[t]{\downarrow}$ denotes that the $e$-th Turing machine, which computes $\varphi_{e}$, stops after at most $t$ steps on input $n$. Then we can define the sets $W_e := \dom(\varphi_e)$ and $W_e[t] := \{n < t \mid \varphi_{e}(n)[t] \downarrow \}$ for all $e, t \in \IN$. Clearly, $\bigcup_{t \in \IN} W_e[t] = W_e$.

Since we are using dyadic series most of the time, we will only consider positive real numbers for the remainder of this article.

\section{Reordered Computable Numbers}\label{sec:reordered-computable-numbers}

It is easy to see that a function $f \colon \IN\to \IN$ tends to infinity if and only if the set $\{k \in \IN \mid f(k) = n\}$ is finite for every $n \in \IN$. This characterization allows us the following definition which will be used all over this article.

\begin{defi}\label{def:definition-u_f}
	Let $f \colon \IN \to \IN$ be a function which tends to infinity. We define the function $u_f \colon \IN\to\IN$ with
	\begin{equation*}
		u_f(n) := \left|\{k \in \IN \mid f(k) = n\}\right|
	\end{equation*}
	for all $n \in \IN$.
\end{defi}

In other words, $u_f$ returns for some input $n$ how often $n$ is enumerated by $f$. Note that in general $u_f$ is not computable, even if $f$ is computable. One important property of this operator is its invariance under permutations. The following statements are obvious. We omit the proofs.

\begin{prop}\label{prop:eigenschaften-u_f}
	Let $f \colon \IN \to \IN$ be a function which tends to infinity.
	\begin{enumerate}[(a)]
		\item For every bijective function $\sigma \colon \IN \to \IN$, the composition $f \circ \sigma$ also tends to infinity and we have $u_{f \circ \sigma} = u_f$.
		\item There exists a unique non-decreasing and unbounded function $f^\ast \colon \IN \to \IN$ with $u_{f^\ast} = u_f$.
		\item For every function $g \colon \IN \to \IN$ tending to infinity and satisfying $u_g = u_f$, there exists a bijective function $\sigma \colon \IN\to\IN$ with $f \circ \sigma = g$.
	\end{enumerate}
\end{prop}

Likewise $u_f$, the function $f^\ast$ is not computable in general, even if $f$ is computable.
Now suppose that $f \colon \IN \to \IN$ is a function such that the series $\sum_{k=0}^{\infty} 2^{-f(k)}$ converges. Of course, this implies that $f$ tends to infinity. From classical analysis we know that for every bijective function $\sigma \colon \IN \to \IN$ the rearranged series $\sum_{k=0}^{\infty} 2^{-f(\sigma(k))}$ also converges, towards the same limit. However, the speed of convergence of the rearranged series may differ from the speed of convergence of the original series. Clearly, a rearranged series converges fastest if its biggest jumps are made first.

\begin{prop}\label{prop:reordered-series-optimal-convergence}
	Let $f \colon \IN\to \IN$ be a function such that the series $\sum_{k=0}^{\infty} 2^{-f(k)}$ converges. Let $\sigma \colon \IN \to \IN$ be a bijective function. Then we have 
	\begin{equation*}
		\sum_{k=n}^{\infty} 2^{-f^\ast(k)} \leq \sum_{k=n}^{\infty} 2^{-f(\sigma(k))}
	\end{equation*}
	for all $n \in \IN$.
\end{prop}

Proposition \ref{prop:reordered-series-optimal-convergence} shows us that the convergence speed of a series can be accelerated if it is rearranged and that there exists an optimal rearrangement of that series. This has already a significant impact, as the following example shows. Let $x$ be a strongly left-computable number which is not computable, and let $f \colon \IN \to \IN$ be a computable injective function with $\sum_{k=0}^{\infty} 2^{-f(k)} = x$. Surely, this series does not converge computably. However, $f^\ast$ is an increasing function, and the reordered series $\sum_{k=0}^{\infty} 2^{-f^\ast(k)}$ converges computably.

Another attempt to increase the convergence speed of a dyadic series after the reordering is by grouping jumps of the same size. Suppose that $f \colon \IN\to \IN$ is a function such that the series $\sum_{k=0}^{\infty} 2^{-f(k)}$ converges. In most cases, the series $\sum_{k=0}^{\infty} u_f(k) \cdot 2^{-k}$ should converge faster than the series $\sum_{k=0}^{\infty} 2^{-f^\ast(k)}$, especially if $f$ enumerates many numbers several times. However, the gap between a non-computable convergence and a computable convergence of a reordered series cannot be overcome by grouping. To show this, we first prove the following lemma.

\begin{lem}\label{lem:identity-added}
	Let $f, g \colon \IN \to \IN$ be non-decreasing and unbounded functions with $u_g(n) = u_f(n) + 1$ for all $n \in \IN$. The series $\sum_{k=0}^{\infty} 2^{-f(k)}$ converges computably if and only if the series $\sum_{k=0}^{\infty} 2^{-g(k)}$ converges computably.
\end{lem}

\begin{proof}
	We will only show the non-obvious implication. Suppose that the series $\sum_{k=0}^{\infty} 2^{-f(k)}$ converges computably, and fix some computable function $s \colon \IN \to \IN$ which is a modulus of convergence.
	
	We define the function $t \colon \IN\to\IN$ by $t(n):=2 \cdot s(n+1) + n + 3$. It is clear that $t$ is computable. We still have to show that $t$ is a modulus of convergence of the series $\sum_{k=0}^{\infty} 2^{-g(k)}$. Note that every natural number is enumerated by $g$ at least once.
	
	Let $n \in \IN$ be arbitrary. We have to distinguish two different cases:
	\begin{enumerate}[(1)]
		\item Among the first $t(n)$ elements, more than $s(n+1)$ have been enumerated by $g$ more than once. This implies that $f$ enumerates those elements at least once as well. Furthermore, the number $f(s(n+1))$ has been enumerated by $g$.
		Note that, due to $\sum_{k=s(n+1)}^{\infty} 2^{-f(k)} \leq 2^{-(n+1)}$, we have $f(k) > n+1$ for all $k \geq s(n)$ in particular. Since $f$ is non-decreasing, all numbers less than or equal to $n+1$ have also been enumerated by $g$.
		We conclude
		\begin{equation*}
			\sum_{k=t(n)}^\infty 2^{-g(k)}
			\leq \sum_{k=s(n+1)+1}^\infty 2^{-f(k)} + \sum_{k=n+2}^\infty 2^{-k}
			\leq 2^{-(n+1)} + 2^{-(n+1)}
			= 2^{-n} .
		\end{equation*}
		\item Among the first $t(n)$ elements, at most $s(n+1)$ have been enumerated by $g$ more than once. Let $\ell \in \{0, \dots, s(n+1)\}$ be the exact number. In other words, among the first $t(n)$ elements, at least $t(n)-\ell \geq t(n)-s(n+1) = s(n+1) + n + 3$ have been enumerated by $g$ only once. Since $g$ is non-decreasing, at least all natural numbers up to $m := s(n+1) + n + 2$ have been enumerated by $g$. After the first $\ell$ elements being enumerated by $f$, the next $s(n+1) - \ell$ elements have to be greater than $m$. Using $k \cdot 2^{-k} \leq \frac{1}{2}$ for all $k\in\IN$, we obtain
		\begin{align*}
			\sum_{k=t(n)}^\infty 2^{-g(k)}
			&\leq \sum_{k=s(n+1)}^\infty 2^{-f(k)} 
			+ (s(n+1)-\ell) \cdot 2^{-(m+1)}
			+ \sum_{k=m+1}^\infty 2^{-k} \\
			&\leq  2^{-(n+1)} + s(n+1) \cdot 2^{-(m+1)} + 2^{-m} \\
			&\leq 2^{-(n+1)} + 2^{-(n+4)} + 2^{-(n+2)} \\
			&< 2^{-n} . 
		\end{align*}	
	\end{enumerate}
	Therefore, the series $\sum_{k=0}^{\infty} 2^{-g(k)}$ converges computably.
\end{proof}

\begin{prop}\label{prop:characterisierung-roc-number}
	Let $f \colon \IN \to \IN$ be a function which tends to infinity. Then the following statements are equivalent:
	\begin{enumerate}[(1)]
		\item There is a bijective function $\sigma \colon \IN \to \IN$ such that the series $\sum_{k=0}^{\infty} 2^{-f(\sigma(k))}$ converges computably.
		\item The series $\sum_{k=0}^{\infty} 2^{-f^\ast(k)}$ converges computably.
		\item The series $\sum_{k=0}^{\infty} u_f(k) \cdot 2^{-k}$ converges computably.
	\end{enumerate}
\end{prop}

\begin{proof} \;
	\begin{itemize}
		\item[$(1) \implies (2)$:] This is clear due to Proposition \ref{prop:reordered-series-optimal-convergence}.
		\item[$(2) \implies (1)$:] This is clear due to Proposition \ref{prop:eigenschaften-u_f}.
		\item[$(2) \implies (3)$:] Suppose that the series $\sum_{k=0}^{\infty} 2^{-f^\ast(k)}$ converges computably, and let $s \colon \IN\to \IN$ be a computable modulus of convergence. Due to Lemma \ref{lem:identity-added}, we can assume $u_f(n) \geq 1$ for all $n \in \IN$. Since $f^\ast$ is non-decreasing, we have $f^\ast(n) \leq n$ for all $n \in \IN$ in particular. Considering an arbitrary $n \in \IN$, we finally obtain
		\begin{equation*}
			\sum_{k=s(n)}^{\infty} u_f(k) \cdot 2^{-k} \leq \sum_{k=s(n)}^{\infty} 2^{-f^\ast(k)} \leq 2^{-n}.
		\end{equation*}
		Therefore, the series $\sum_{k=0}^{\infty} u_f(k) \cdot 2^{-k}$ converges computably.
		\item[$(3) \implies (2)$:] Suppose that the series $\sum_{k=0}^{\infty} u_f(k) \cdot 2^{-k}$ converges computably, and let $s \colon \IN \to \IN$ be a computable modulus of convergence. Fix some $c \in \IN$ with $\sum_{k=0}^{\infty} u_f(k) \cdot 2^{-k} \leq 2^c$. Then we must have $u_f(n) \leq 2^{c+n}$ for all $n \in \IN$. Define the function $t \colon \IN\to \IN$ by $t(n) := \sum_{k=0}^{s(n)-1} 2^{c+k}$. Obviously, $t$ is computable. Considering an arbitrary $n \in \IN$, we finally obtain
		\begin{equation*}
			\sum_{k=t(n)}^{\infty} 2^{-f^\ast(k)} \leq \sum_{k=s(n)}^{\infty} u_f(k) \cdot 2^{-k} \leq 2^{-n}.
		\end{equation*}
	Therefore, the series $\sum_{k=0}^{\infty} 2^{-f^\ast(k)}$ converges computably.
	\end{itemize}
\end{proof}

For the remainder of this article we are only interested in functions $f \colon \IN \to \IN$ such that the series $\sum_{k=0}^{\infty} 2^{-f(k)}$ converges.

\begin{defi}\label{def:name-of-a-real-number}
	Let $f \colon \IN \to \IN$ be a function and $x$ be a real number. We say that $f$ is a \emph{name} for $x$ if we have $\sum_{k=0}^{\infty} 2^{-f(k)} = x$.
\end{defi}

Obviously, a name uniquely describes a real number. It is also clear that every real number has infinitely many names. Furthermore, a real number is left-computable if and only if it has a computable name. Now we are ready for the central definition of this article.

\begin{defi}\label{def:reordered-computable-number}
	We say that a real number is \emph{reordered computable} if it has a computable name $f \colon \IN \to \IN$ such that the series $\sum_{k=0}^{\infty} u_f(k) \cdot 2^{-k}$ converges computably.
\end{defi}

By definition, every reordered computable number is left-computable. In Section~\ref{sec:counterexamples} we will see that the converse is not true. At the end of this section, we will give a simple example for a reordered computable number.

\begin{defi}[Wu \cite{Wu05a}]
	A real number is called \emph{regular} if it can be written as the sum of strongly left-computable numbers.
\end{defi}

Every regular number is left-computable, but the converse is not true. For example, every left-computable number which is Martin-Löf random is not regular. It is well-known that strongly left-computable numbers are not Martin-Löf random, and Downey, Hirschfeldt and Nies showed \cite{DHN2002} that the sum of left-computable numbers which are not Martin-Löf random is not Martin-Löf random as well. Therefore, regular numbers are not Martin-Löf random.
The following characterization for regular numbers is easy to verify. We omit the proof.

\begin{prop}\label{prop:charakterisierung-regular}
	For a real number $x$ the following are equivalent:
	\begin{enumerate}[(1)]
		\item $x$ is regular.\label{prop:charakterisierung-regular-property-1}
		\item There exists a computable name $f \colon \IN \to \IN$ for $x$ such that $u_f$ is bounded. \label{prop:charakterisierung-regular-property-2}
	\end{enumerate}
\end{prop}

\begin{prop}
	Every regular number is reordered computable.
\end{prop}

\begin{proof}
	Let $x$ be a regular number. According to Proposition \ref{prop:charakterisierung-regular},
	there exist a computable name $f \colon \IN \to\IN$ for $x$ and a constant $c \in \IN$ with $u_f(n) \leq 2^c$ for all $n \in \IN$. Define the function $s \colon \IN \to \IN$ with $s(n) := n + c + 1$ for all $n \in \IN$. Surely, $s$ is computable, and we claim that $s$ is a modulus of convergence for the series $\sum_{k=0}^{\infty} u_f(k) \cdot 2^{-k}$. Considering an arbitrary $n \in \IN$, we obtain
	\begin{equation*}
		\sum_{k=s(n)}^{\infty} u_f(k) \cdot 2^{-k} \leq \sum_{k=s(n)}^{\infty} 2^c \cdot 2^{-k} = 2^c \cdot \sum_{k=n+c+1}^{\infty} 2^{-k} = 2^c \cdot 2^{-(n+c)} = 2^{-n}.
	\end{equation*}
\end{proof}

Later we will see that not every reordered computable number is regular. Further, more abstract examples will also occur in the next sections.

\section{Calculating with Limits}\label{sec:calculating-with-limits}

This section contains the proofs of several technical lemmas. They have been outsourced in order to improve the readability of the next sections where these results are used. The first lemma in this section is the well-known root test from classical calculus, which also gives us an effective estimate for the remainder.

\begin{lem}\label{lem:wurzelkriterium-abschaetzung-restglied}
	Let $(a_n)_n$ be a sequence with $\limsup_{n\to\infty} \sqrt[n]{\left| a_n \right|} < 1$. Then the series $\sum_{k=0}^{\infty} a_k$ converges absolutely, and there is a rational number $q \in \left[0, 1\right[ $ with $\sum_{k=n}^{\infty} \left|a_k\right| \leq \frac{q^{n}}{1-q}$ for all $n \in \IN$.
\end{lem}

\begin{proof}
	Since we have $\limsup_{n\to\infty} \sqrt[n]{\left| a_n \right|} < 1$, there exist a real number ${p \in \left[0, 1\right[}$ and a number $m \in \IN$ with $\left|a_n\right| \leq p^n$ for all $n \geq m$. So we have $\sum_{k=n}^{\infty} \left|a_k\right| \leq \sum_{k=n}^{\infty} p^k = \frac{p^n}{1-p}$
	for all $n \geq m$. This already implies that $\sum_{k=0}^{\infty} a_k$ converges absolutely. Furthermore, fix some rational number $q \in \left[p, 1\right[$ with $\sum_{k=n}^{\infty} \left|a_k\right| \leq \frac{q^n}{1-q}$ for all $n < m$. Note that such a number exists, since we have $\lim_{q \to 1, q < 1} \frac{q^n}{1-q} = \infty$ for any fixed $n \in \IN$, and there are only finitely many $n < m$. So we can finally conclude $\sum_{k=n}^{\infty} \left|a_k\right| \leq \frac{q^n}{1-q}$ for all $n \in \IN$.
\end{proof}

One application of Lemma \ref{lem:wurzelkriterium-abschaetzung-restglied} is the proof of the next lemma, which gives us a sufficient criterion for computable convergence of a dyadic series. Note that there are no assumptions on the computability of the function.

\begin{lem}\label{lem:berechenbare-konvergenz-dyadische-reihe}
	Let $f \colon \IN \to \IN$ be a function with $\limsup_{n\to\infty} \sqrt[n]{f(n)} < 2$. Then the series $\sum_{k=0}^{\infty} f(k) \cdot 2^{-k}$ converges computably.
\end{lem}

\begin{proof}
	$\limsup_{n\to\infty} \sqrt[n]{f(n)} < 2$  implies $\limsup_{n\to\infty} \sqrt[n]{f(n) \cdot 2^{-n}} < 1$. Due to Lemma~\ref{lem:wurzelkriterium-abschaetzung-restglied}, the series $\sum_{k=0}^{\infty} f(k) \cdot 2^{-k}$ converges, and there is a rational number $q \in \left[0,1\right[$ with $\sum_{k=n}^{\infty} f(k) \cdot 2^{-k} \leq \frac{q^{n}}{1-q}$ for all $n \in \IN$. Define the function $s \colon \IN \to \IN$ with $ 
	s(n) := \min\left\{m \in \IN \mid \frac{q^{m}}{1-q} \leq 2^{-n} \right\}
	$ 
	for all $n \in \IN$. Clearly, $s$ is computable, and we claim that $s$ is a modulus of convergence for the series. Considering an arbitrary $n \in \IN$, we obtain 
	\begin{equation*}
		\sum_{k=s(n)}^{\infty} f(k) \cdot 2^{-k} \leq \frac{q^{s(n)}}{1-q} \leq 2^{-n}
	\end{equation*}
	by the definition of $s$.
\end{proof}

The following two lemmas are very important, since they are used several times in the next sections.

\begin{lem}\label{lem:limsup-g-summe-f}
	Let $f : \IN \to \IN$ be a function satisfying one of these conditions:
	\begin{itemize}
		\item $f(n) = 0$ for all $n \in \IN$
		\item $f(n) \geq 1$ for infinitely many $n \in \IN$
	\end{itemize}
	Furthermore, let $c \in \IN$ be a constant and $g \colon \IN \to \IN$ be the function defined by $g(n) := \sum_{k=0}^{n+c} f(k)$ for all $n \in \IN$. Then the following statements hold:
	\begin{enumerate}[(a)]
		\item $\limsup_{n\to\infty} \sqrt[n]{g(n)} = \limsup_{n\to\infty} \sqrt[n]{f(n)}$ \label{statement-1}
		\item If $\lim_{n\to\infty} \sqrt[n]{f(n)}$ exists, we have $\lim_{n\to\infty} \sqrt[n]{g(n)} = \lim_{n\to\infty} \sqrt[n]{f(n)}$. \label{statement-2}
	\end{enumerate}
\end{lem}

\begin{proof}
	Both statements are obvious if $f$ meets the first condition. Therefore, suppose that $f$ satisfies the second condition. Due to $f(n) \leq g(n)$ for all $n \in \IN$, it is clear that the following inequalities hold:
	\begin{align}
		\liminf_{n\to\infty} \sqrt[n]{f(n)} &\leq \liminf_{n\to\infty} \sqrt[n]{g(n)} \label{fact:abschaetzung-liminf-summe} \\
		\limsup_{n\to\infty} \sqrt[n]{f(n)} &\leq \limsup_{n\to\infty} \sqrt[n]{g(n)} \label{fact:abschaetzung-limsup-summe}
	\end{align}	
	\begin{enumerate}[(a)]
		\item We only show $\limsup_{n\to\infty} \sqrt[n]{g(n)} \leq  \limsup_{n\to\infty} \sqrt[n]{f(n)}$, since the other direction is clear, due to Inequality \ref{fact:abschaetzung-limsup-summe}. First we get
		\begin{align*}
			\limsup_{n\to\infty} \sqrt[n]{g(n)} &= \limsup_{n\to\infty} \sqrt[n]{\sum_{k=0}^{n+c} f(k)} \\
			&\leq \limsup_{n\to\infty} \sqrt[n]{(n+c+1) \max_{k\in \{0, \dots, n\}} f(k)} \\
			&= \limsup_{n\to\infty} \sqrt[n]{\max_{k\in \{0, \dots, n\}} f(k)}.
		\end{align*}
		The last identity is due to $\lim_{n\to\infty} \sqrt[n]{n+c+1} = 1$. Let us perform a case distinction:
		\begin{itemize}
			\item $f$ is bounded. Let $m := \max_{k\in \IN} f(k)$. Then we have both $m \geq 1$ and $\limsup_{n\to\infty} \sqrt[n]{f(n)} = 1$, due to $f(n) \geq 1$ for infinitely many $n \in \IN$. We obtain as desired
			\begin{align*}
				\limsup_{n\to\infty} \sqrt[n]{g(n)} &\leq \limsup_{n\to\infty} \sqrt[n]{\max_{k\in \{0, \dots, n\}} f(k)} \\
				&= \limsup_{n\to\infty} \sqrt[n]{m} \\
				&= 1 \\
				&= \limsup_{n\to\infty} \sqrt[n]{f(n)}.
			\end{align*}
			\item $f$ is unbounded. In this case there exist infinitely many $n \in \IN$ with $\max_{k\in \{0, \dots, n\}} f(k) = f(n)$. We obtain as desired
			\begin{align*}
				\limsup_{n\to\infty} \sqrt[n]{g(n)} &\leq \limsup_{n\to\infty} \sqrt[n]{\max_{k\in \{0, \dots, n\}} f(k)} \\
				&= \limsup_{n\to\infty} \sqrt[n]{f(n)}.
			\end{align*}
		\end{itemize}
		\item It suffices to show $\limsup_{n\to\infty} \sqrt[n]{g(n)} \leq \liminf_{n\to\infty} \sqrt[n]{g(n)}$. We obtain:
		\begin{align*}
			\limsup_{n\to\infty} \sqrt[n]{g(n)}
			&= \limsup_{n\to\infty} \sqrt[n]{f(n)} &\text{Statement \ref{statement-1}} \\
			&= \liminf_{n\to\infty} \sqrt[n]{f(n)} &\text{$\lim_{n\to\infty} \sqrt[n]{f(n)}$ exists} \\
			&\leq \liminf_{n\to\infty} \sqrt[n]{g(n)} &\text{Inequality \ref{fact:abschaetzung-liminf-summe}}
		\end{align*}
	\end{enumerate}	
\end{proof}

It is worth mentioning that an analogous statement for the limit inferior does not hold. For example, choose the function $f\colon\IN\to\IN$ defined by
\begin{equation*}
	f(n) := \begin{cases}
		2^n &\text{if $n$ is even} \\
		1 &\text{otherwise}
	\end{cases}
\end{equation*}
and the function $g\colon\IN\to\IN$ defined by $g(n) := \sum_{k=0}^{n} f(k)$. Then we have $\liminf_{n\to\infty} \sqrt[n]{f(n)} = 1$, but $\liminf_{n\to\infty} \sqrt[n]{g(n)} = 2$.

\begin{lem}\label{lem:konvergenz-g-summe-f}
	Let $f \colon \IN \to \IN$ be a function such that the series ${\sum_{k=0}^{\infty} f(k) \cdot 2^{-k}}$ converges. Let $c \in \IN$, and define the function $g \colon \IN\to\IN$ with $g(n) := \sum_{k=0}^{n+c} f(k)$ for all $n \in \IN$. Then we have
	\begin{equation}\label{eq:restglied-g-sum-f}
		\sum_{k=n}^{\infty} g(k) \cdot 2^{-k} = 2^{-n+1} \cdot \sum_{k=0}^{n+c-1} f(k) + 2^{c+1} \cdot \sum_{k=n+c}^{\infty} f(k)\cdot 2^{-k}
	\end{equation}
	for all $n \in \IN$, hence the series $\sum_{k=0}^{\infty} g(k) \cdot 2^{-k}$ also converges.
	Furthermore, if $\sum_{k=0}^{\infty} f(k) \cdot 2^{-k}$ converges computably, then $\sum_{k=0}^{\infty} g(k) \cdot 2^{-k}$ converges computably as well.
\end{lem}

\begin{proof}
	Considering an arbitrary $n \in \IN$, we obtain
	\begin{align*}
		\sum_{k=n}^{\infty} g(k) \cdot 2^{-k}	
		&= \sum_{k=n}^{\infty} \sum_{j=0}^{k+c} f(j) \cdot 2^{-k} \\
		&= \sum_{k=0}^{n+c-1} f(k) \cdot \sum_{j=n}^{\infty} 2^{-j} + \sum_{k=n+c}^{\infty} f(k) \cdot \sum_{j=k}^{\infty} 2^{-(j-c)} \\
		&= 2^{-n+1} \cdot \sum_{k=0}^{n+c-1} f(k) + 2^{c+1} \cdot \sum_{k=n+c}^{\infty} f(k)\cdot 2^{-k}.
	\end{align*}
	Now suppose that the series $\sum_{k=0}^{\infty} f(k) \cdot 2^{-k}$ converges computably. Fix some computable increasing function $r \colon \IN \to \IN$ which is a modulus of convergence for the series $\sum_{k=0}^{\infty} f(k) \cdot 2^{-k}$, and fix some natural number $d \geq c$ with $\sum_{k=0}^{\infty} f(k) \cdot 2^{-k} \leq 2^{d}$. We claim that the function $s\colon\IN\to\IN$ defined by $s(n) := 2\cdot r(n+d+2)$ is a modulus of convergence for the series $\sum_{k=0}^{\infty} g(k) \cdot 2^{-k}$.
	
	First we observe
	\begin{align*}
		2^{-s(n)+1} \cdot \sum_{k=0}^{r(n+d+2)-1} f(k)
		&= 2^{-s(n)+1} \cdot \sum_{k=0}^{r(n+d+2)-1} 2^k \cdot f(k) \cdot 2^{-k} \\
		&\leq 2^{-s(n)+1} \cdot \sum_{k=0}^{r(n+d+2)-1} 2^{r(n+d+2)-1} \cdot f(k) \cdot 2^{-k} \\
		&= 2^{-r(n+d+2)} \cdot \sum_{k=0}^{r(n+d+2)-1} f(k) \cdot 2^{-k} \\
		&\leq 2^{-(n+d+2)} \cdot 2^{d} \\
		&= 2^{-(n+2)}
	\end{align*}
	for all $n \in \IN$ and
	\begin{align*}
		2^{-s(n)+1} \cdot \sum_{k=r(n+d+2)}^{s(n)+d-1} f(k)
		&= 2^{-s(n)+1} \cdot \sum_{k=r(n+d+2)}^{s(n)+d-1} 2^k \cdot f(k) \cdot 2^{-k} \\
		&\leq 2^{-s(n)+1} \cdot \sum_{k=r(n+d+2)}^{s(n)+d-1} 2^{s(n)+d-1} \cdot f(k) \cdot 2^{-k} \\
		&= 2^{d} \cdot \sum_{k=r(n+d+2)}^{s(n)+d-1} f(k) \cdot 2^{-k} \\
		&\leq 2^{d} \cdot 2^{-(n+d+2)} \\
		&= 2^{-(n+2)}
	\end{align*}
	for all $n \in \IN$. Combining these two inequalities, we obtain
	\begin{equation}\label{ineq:abschaetzung-prolog}
		2^{-s(n)+1} \cdot \sum_{k=0}^{s(n)+d-1} f(k) \leq 2^{-(n+1)}
	\end{equation}
	for all $n \in \IN$. Using Equation \ref{eq:restglied-g-sum-f} and Inequality \ref{ineq:abschaetzung-prolog}, we finally get
	\begin{align*}
		\sum_{k=s(n)}^{\infty} g(k) \cdot 2^{-k}
		&= 2^{-s(n)+1} \cdot \sum_{k=0}^{s(n)+c-1} f(k) + 2^{c+1} \cdot \sum_{k=s(n)+c}^{\infty} f(k)\cdot 2^{-k} \\
		&\leq 2^{-s(n)+1} \cdot \sum_{k=0}^{s(n)+d-1} f(k) + 2^{d+1} \cdot \sum_{k=s(n)}^{\infty} f(k)\cdot 2^{-k} \\
		&\leq 2^{-(n+1)} + 2^{-(n+1)} \\
		&= 2^{-n}
	\end{align*}
	for all $n \in \IN$. Since $s$ is computable, the series $\sum_{k=0}^{\infty} g(k) \cdot 2^{-k}$ converges computably.
\end{proof}

In contrast to the previous two lemmas, the next lemma is easy to prove. It is used in Theorem~\ref{satz:hierarchiesatz-roc}, which is our last main result of this article, and it is the last lemma in this section.

\begin{lem}\label{prop:berechenbare-konvergenz-berechenbar-approximierbar}
	For every computably approximable number $\rho \in \left[1, 2\right]$ there exists a computable function $r \colon \IN \to \IN$ which tends to infinity and satisfies the following properties:
	\begin{enumerate}[(1)]
		\item The series $\sum_{k=0}^{\infty} r(k) \cdot 2^{-k}$ converges computably.
		\item The sequence $\left(\sqrt[n]{r(n)}\right)_n$ converges to $\rho$.
	\end{enumerate}
\end{lem}

\begin{proof}
	In order to define $r \colon \IN \to \IN$, we consider three distinct cases:
	\begin{itemize}
		\item If $\rho = 1$, we define $r(n) := n$ for all $n \in \IN$.
		\item If $\rho = 2$, we define
		\begin{equation*}
			r(n) := \begin{cases}
				0 &\text{if $n = 0$} \\
				\lceil \frac{2^n}{n^2} \rceil &\text{otherwise}
			\end{cases}
		\end{equation*}
		for all $n \in \IN$.
		\item In all other cases we fix a computable non-negative sequence of rational numbers $(\rho_n)_n$ converging to $\rho$ and define $r(n) := \lceil \rho_n^n \rceil$ for all $n \in \IN$.
	\end{itemize}
	In all cases, $r$ is computable and tends to infinity, the series $\sum_{k=0}^{\infty} r(k) \cdot 2^{-k}$ converges computably and we have $\lim_{n\to\infty} \sqrt[n]{r(n)} = \rho$.
\end{proof}

\section{Closure properties}\label{sec:closure-properties}

In this section we will show that the reordered computable numbers are closed under addition and multiplication, and we will also show that they are closed downwards under the Solovay reduction.

\begin{satz}\label{prop:closure-sum}
	Let $x$ and $y$ be reordered computable numbers. Then the sum $x + y$ is also reordered computable.
\end{satz}

\begin{proof}
	Fix computable names $f, g \colon \IN \to \IN$ witnessing the reordered computability of $x$ and $y$, respectively. Also fix computable functions $r, s \colon \IN \to \IN$ which are a modulus of convergence for the series $\sum_{k=0}^{\infty} u_f(k) \cdot 2^{-k}$ and  $\sum_{k=0}^{\infty} u_g(k) \cdot 2^{-k}$, respectively. Define the function $h \colon \IN \to \IN$ with $h(2n) := f(n)$ and $h(2n+1) := g(n)$ for all $n \in \IN$. Obviously, $h$ is a computable name for $x + y$. We also observe $u_h(n) = u_f(n) + u_g(n)$ for all $n \in \IN$. Define the function $t \colon \IN \to \IN$ with $t(n) := \max\{r(n+1), s(n+1)\}$ for all $n \in \IN$. Since $r$ and $s$ are both computable, $t$ is computable as well. We claim that $t$ is a modulus of convergence for the series $\sum_{k=0}^{\infty} u_{h}(k) \cdot 2^{-k}$. Considering some $n \in \IN$, we finally obtain
	\begin{equation*}
		\sum_{k=t(n)}^{\infty} u_{h}(k) \cdot 2^{-k} 
		= \sum_{k=t(n)}^{\infty} u_{f}(k) \cdot 2^{-k} + \sum_{k=t(n)}^{\infty} u_{g}(k) \cdot 2^{-k} \\
		\leq 2^{-(n+1)} + 2^{-(n+1)}  \\
		= 2^{-n}.
	\end{equation*}
	Thus, $x + y$ is a reordered computable number.
\end{proof}

\begin{satz}\label{prop:closure-product}
	Let $x$ and $y$ be reordered computable numbers. Then the product $x \cdot y$ is also reordered computable.
\end{satz}

\begin{proof}
	Fix computable functions $f, g \colon \IN \to \IN$ witnessing the reordered computability of $x$ and $y$, respectively, and computable functions $r, s \colon \IN \to \IN$ which are a modulus of convergence for the reordered series $\sum_{k=0}^{\infty} 2^{-f^\ast(k)}$ and $\sum_{k=0}^{\infty} 2^{-g^\ast(k)}$, respectively. From the Cauchy product formula
	\begin{equation}
		x \cdot y = \sum_{k=0}^{\infty} 2^{-f(k)} \cdot \sum_{k=0}^{\infty} 2^{-g(k)} = \sum_{\ell=0}^{\infty} \sum_{k=0}^{\ell} 2^{-f(k)} \cdot 2^{-g(\ell-k)}
	\end{equation}
	one can construct a computable function $h \colon \IN \to\IN$ with $\sum_{k=0}^{\infty} 2^{-h(k)} = x \cdot y$.
	We also observe $u_h(n) = \sum_{k=0}^{n} u_f(k) \cdot u_g(n-k)$ for all $n \in \IN$.
	Note that we also have $\sum_{\ell=0}^{\infty} \sum_{k=0}^{\ell} 2^{-f^\ast(k)} \cdot 2^{-g^\ast(\ell-k)} = x \cdot y$.  Fix some constant $c \in \IN$ with $2^{c} \geq \max\{x,y\}$.
	First we claim that we have
	\begin{equation}
		\sum_{k=0}^{\ell} 2^{-f^\ast(k)} \cdot  2^{-g^\ast(\ell-k)} \leq 2^c \cdot \left(2^{-f^\ast\left(\ell/2\right)} + 2^{-g^\ast\left(\ell/2\right)} \right)
	\end{equation}
	for all $\ell \in \IN$. Considering some $\ell \in \IN$, we obtain
	\begin{align*}
		\sum_{k=0}^{\ell} 2^{-f^\ast(k)} \cdot  2^{-g^\ast(\ell-k)}
		&= \sum_{k=0}^{\ell/2} 2^{-f^\ast(k)} \cdot  2^{-g^\ast(\ell-k)} + \sum_{k=\ell/2 + 1}^{\ell} 2^{-f^\ast(k)} \cdot  2^{-g^\ast(\ell-k)} \\
		&\leq \sum_{k=0}^{\ell/2} 2^{-f^\ast(k)} \cdot  2^{-g^\ast\left(\ell/2\right)} + \sum_{k=\ell/2 + 1}^{\ell} 2^{-f^\ast\left(\ell/2\right)} \cdot  2^{-g^\ast(\ell-k)} \\ 
		&= 2^{-g^\ast\left(\ell/2\right)} \cdot  \sum_{k=0}^{\ell/2} 2^{-f^\ast(k)} + 2^{-f^\ast\left(\ell/2\right)} \cdot \sum_{k=\ell/2 + 1}^{\ell}   2^{-g^\ast(\ell-k)} \\
		&\leq 2^{-g^\ast\left(\ell/2\right)} \cdot  2^c + 2^{-f^\ast\left(\ell/2\right)} \cdot 2^c \\
		&= 2^c \cdot \left(2^{-f^\ast\left(\ell/2\right)} + 2^{-g^\ast\left(\ell/2\right)} \right).
	\end{align*}
	Finally we claim that the function $t \colon \IN \to \IN$ defined by
	\begin{equation*}
		t(n) := 2\cdot \max\{r(n+c+2), s(n+c+2)\}
	\end{equation*}
	is a modulus of convergence of the series $\sum_{k=0}^{\infty} 2^{-h^\ast(k)}$. According to Proposition~\ref{prop:characterisierung-roc-number}, it is sufficient to show computable convergence for some arbitrary rearrangement of this series. So we get
	\begin{align*}
		\sum_{\ell=t(n)}^{\infty} \sum_{k=0}^{\ell} 2^{-f^\ast(k)} \cdot  2^{-g^\ast(l-k)} 
		&\leq \sum_{\ell=t(n)}^{\infty} 2^c \cdot \left(2^{-f^\ast\left(\ell/2\right)} + 2^{-g^\ast\left(\ell/2\right)} \right) \\
		&\leq 2^c \cdot \left(\sum_{\ell=t(n)}^{\infty}  2^{-f^\ast\left(\ell/2\right)} + \sum_{\ell=t(n)}^{\infty}  2^{-g^\ast\left(\ell/2\right)}\right) \\
		&\leq 2^c \cdot \left(2 \cdot \sum_{\ell=t(n)/2}^{\infty} 2^{-f^\ast(\ell)} + 2\cdot \sum_{\ell=t(n)/2}^{\infty} 2^{-g^\ast(\ell)}\right) \\
		&\leq 2^c \cdot \left(2\cdot 2^{-(n+c+2)} + 2\cdot 2^{-(n+c+2)}\right) \\
		&= 2^c \cdot 2^2 \cdot 2^{-(n+c+2)} \\
		&= 2^{-n}.
	\end{align*}
	Since $t$ is also computable, $x \cdot y$ is a reordered computable number.
\end{proof}

Finally, we show that reordered computable numbers are closed downwards under the Solovay reduction, which is a relation defined on the set of the left-computable numbers. We will use a simplified definition. The proof for the equivalence with the original definition can be found in \cite{CHKW01}.

\begin{defi}[Solovay \cite{Sol1975}]
	Let $x$ and $y$ be left-computable numbers. It is said that $x$ is \emph{Solovay reducible} to $y$ (written $x \leq_S y$ for short) if there exist computable increasing sequences of rational numbers $(x_n)_n$ and $(y_n)_n$ converging to $x$ and $y$, respectively, and a constant $c > 0$ with
	\begin{equation*}
		x - x_{n} \leq c \cdot \left(y - y_n\right)
 	\end{equation*}
 	for all $n \in \IN$.
\end{defi}

It is obvious that the Solovay reduction is reflexive and it is easy to verify that it is transitive. Therefore, this relation is a pre-order. 
If we have both $x \leq_S y$ and $y \leq_S x$, we write $x \equiv_S y$. This is an equivalence relation and the equivalence classes are called \emph{Solovay degrees}. Furthermore, this partial order is an upper semilattice. For left-computable numbers $\alpha$ and $\beta$, the least upper bound of the degrees $\left[\alpha\right]_S$ and $\left[\beta\right]_S$ is simply $\left[\alpha + \beta\right]_S = \left[\alpha \cdot \beta\right]_S$. There is a smallest degree, the computable numbers. There is also a largest degree. Solovay showed \cite{Sol1975} that its elements are Martin-Löf random, and Kucera and Slaman \cite{KS2001} showed that all Martin-Löf random numbers are elements of this degree. Therefore, the Martin-Löf random numbers form the largest degree.

The proof that the reordered computable numbers are closed downwards under the Solovay reduction uses the following lemma by Downey, Hirschfeldt and Nies \cite{DHN2002}.

\begin{lem}[Downey, Hirschfeldt, Nies \cite{DHN2002}]\label{prop:characterisierung-solovay-reduktion}
	Let $x$ and $y$ be left-computable numbers with $x \leq_S y$, and let $(y_n)_n$ be a computable increasing sequence of rational numbers converging to $y$. There exist a computable increasing sequence of rational numbers $(x_n)_n$ converging to $x$ and a constant $c \in \IN$ with
	\begin{equation*}
		x_{n+1} - x_n \leq 2^c \cdot \left(y_{n+1} - y_n\right)
	\end{equation*}
	for all $n \in \IN$.
\end{lem}


\begin{satz}\label{satz:closure-roc-solovay}
	Let $x$ and $y$ be left-computable numbers with $x \leq_S y$. If $y$ is reordered computable, then $x$ is reordered computable as well.
\end{satz}

\begin{proof}
	Fix some computable name $g \colon \IN \to \IN$ witnessing the reordered computability of $y$. According to Lemma \ref{prop:characterisierung-solovay-reduktion}, there exist a computable increasing sequence of rational numbers $(x_n)_n$ converging to $x$ and a constant $c \in \IN$ with $x_{n+1} - x_n \leq 2^{-g(n)+c}$ for all $n \in \IN$. We want to show that $x$ is also reordered computable. The following greedy algorithm recursively computes a function $f \colon \IN \to \IN$ in infinitely many stages such that multiple values can be enumerated in one stage.
	At every stage $t \in \IN$, we enumerate the next $\ell_t$ values for $f$, denoted by $m_{t,1}, \dots, m_{t,\ell_t}$. Our goal is that at the end of stage~$t$ we have 
	$\sum_{s=0}^{t} \sum_{k=1}^{\ell_s} 2^{-m_{s,k}} \in \left]x_{t+1} - 2^{-(t+1)}, x_{t+1} \right]$.
	To achieve this, we one by one enumerate the smallest number which does not yet violate this condition. It is clear that this property will hold after the enumeration of finitely many such numbers. 
%
%
%
%
	Thus, $f$ is well-defined, and we have $\sum_{k=0}^{\infty} 2^{-f(k)} = x$. Therefore, $f$ is a computable name for $x$. 
	
	We still have to show that the series $\sum_{k=0}^{\infty} u_f(k) \cdot 2^{-k}$ converges computably. By construction, the only number which can be enumerated by $f$ for several times during the same stage is $0$. The only stages this can happen are stages when some number $m \in \{0, \dots, c\}$ is enumerated by $g$. 
	Any other number $n \geq 1$ can be enumerated by $f$ in one of these two cases:
	\begin{itemize}
		\item $f$ enumerates $n$ due to the remainder of $x_{t+1} - \sum_{s=0}^{t} \sum_{k=1}^{\ell_s} 2^{-m_{s,k}}$, which is less than $2^{-(t+1)}$ by construction, at some stage $t < n$.
		\item $g$ enumerates some number $m \in \{0, \dots, n+c\}$ at any stage.
	\end{itemize}
	This implies 
	\begin{equation*}
		u_f(n) \leq n + \sum_{k=0}^{n+c} u_g(k)
	\end{equation*}
	 for all $n \geq 1$. Since the series $\sum_{k=0}^{\infty} u_g(k) \cdot 2^{-k}$ converges computably, the series  $\sum_{k=0}^{\infty} u_f(k) \cdot 2^{-k}$ converges computably as well, according to Lemma~\ref{lem:konvergenz-g-summe-f}. Therefore, $x$ is reordered computable.
\end{proof}

\section{Counterexamples}\label{sec:counterexamples}

In this short section we will show that there exist left-computable numbers which are not reordered computable. Our first counterexamples come from the set of left-computable numbers which are nearly computable. The nearly computable numbers are a subset of the real numbers introduced by Hertling and Janicki~\cite{HJ2023}.

\begin{defi}[Hertling, Janicki~\cite{HJ2023}]\label{def:fast-berechenbar}
	\;
	\begin{enumerate}[(1)]
		\item A sequence $(x_n)_n$ is called \emph{nearly computably convergent} if it converges and, for every computable increasing function $s \colon  \IN \to \IN$, the sequence $(x_{s(n+1)} - x_{s(n)})_n$ converges computably to $0$.
		\item A real number is called \emph{nearly computable} if there exists a computable sequence of rational numbers which converges nearly computably to it.
	\end{enumerate}
\end{defi}

Every computable number is nearly computable, but the converse is not true. It follows from a theorem by Downey and LaForte~\cite{DL02} that there exists a left-computable number which is nearly computably but not computable. We will show that those numbers cannot be reordered computable. In order to prove this, we will use the following characterization for left-computable numbers which are nearly computable.

\begin{prop}[Hertling, Janicki~\cite{HJ2023}]\label{prop:characterisierung-lnc}
	For a left-computable number~$x$ the following are equivalent:
	\begin{enumerate}[(1)]
		\item $x$ is nearly computable.
		\item For every computable increasing sequence of rational numbers $(x_n)_n$ converging to $x$, the sequence $(x_{n+1} - x_n)_n$ converges computably to zero.
	\end{enumerate}
\end{prop}

\begin{satz}\label{satz:nc+roc=>computable}
	Let $x$ be a reordered computable number which is nearly computable. Then $x$ is even computable.
\end{satz}

\begin{proof}
	Since $x$ is reordered computable, there exists a computable name $f \colon \IN \to \IN$ for $x$ such that the series $\sum_{k=0}^{\infty} u_f(k) \cdot 2^{-k}$ converges computably. Fix some computable modulus of convergence $s \colon \IN \to \IN$ for this series. Since $x$ is nearly computable as well, the sequence $(2^{-f(n)})_n$ converges computably to zero, according to Proposition \ref{prop:characterisierung-lnc}. Fix some computable modulus of convergence $r \colon \IN \to \IN$ for this sequence. We claim that the computable function $t \colon \IN\to \IN$ defined by $t(n) := r(s(n))$ is a modulus of convergence for the series $\sum_{k=0}^{\infty} 2^{-f(k)}$. Considering some arbitrary $n \in \IN$, we obtain
	\begin{equation*}
		\sum_{k=t(n)}^{\infty} 2^{-f(k)} = \sum_{k=r(s(n))}^{\infty} 2^{-f(k)} \leq \sum_{k=s(n)}^{\infty} u_f(k) \cdot 2^{-k} \leq 2^{-n}.
	\end{equation*}
	Therefore, $x$ is a computable number.
\end{proof}

There are other examples for left-computable numbers which are not reordered computable, namely the Martin-Löf random numbers. Recall that the left-computable numbers which are Martin-Löf random form the largest Solavay degree. Combining this fact with Theorem~\ref{satz:closure-roc-solovay} and Theorem \ref{satz:nc+roc=>computable}, we can conclude that these numbers are not reordered computable either.

\begin{kor}\label{kor:reordered-computable-implies-non-random}
	Every reordered computable number is not Martin-Löf random.
\end{kor}

If follows from a theorem by Stephan and Wu \cite{SW05} that left-computable numbers which are nearly computable are not Martin-Löf random. So there exist left-computable numbers which are neither reordered computable nor nearly computable.

\section{Characteristic Dyadic Values}\label{sec:characteristic-dyadic-values}

In this final section we will investigate how often a computable name of a left-computable number enumerates its elements. Since we are investigating computable dyadic series, a reasonable measure for this might be the value of $\limsup_{n\to\infty} \sqrt[n]{u_f(n)}$ for some computable name $f \colon \IN \to \IN$. The intuition is that the higher this value is, the less efficient is this approximation.

\begin{prop}\label{prop:lim-sup-schranken}
	Let $f \colon \IN \to \IN$ be a function such that the series $\sum_{k=0}^{\infty} 2^{-f(k)}$ converges. Then we have $\limsup_{n\to\infty} \sqrt[n]{u_f(n)} \in \left[1, 2\right]$.
\end{prop}

\begin{proof}
	Since $\sum_{k=0}^{\infty} 2^{-f(k)}$ converges, $f$ tends to infinity. On the one hand, we have $u_f(n) \geq 1$ for infinitely many $n \in \IN$, hence $\limsup_{n\to\infty} \sqrt[n]{u_f(n)} \geq 1$. On the other hand, obviously, the series $\sum_{k=0}^{\infty} u_f(k) \cdot 2^{-k}$ converges as well, hence $\limsup_{n\to\infty} \sqrt[n]{u_f(n)} \leq 2$.
\end{proof}

It is not difficult to show that every left-computable number has a computable name $f \colon \IN \to \IN$ with $\limsup_{n\to\infty} \sqrt[n]{u_f(n)} = 2$. Next we will show that every left-computable number which has a computable name $f \colon \IN \to \IN$ with $\limsup_{n\to\infty} \sqrt[n]{u_f(n)} < 2$ is even reordered computable. 

\begin{prop}\label{prop:limsup<2->roc}
	Let $f \colon \IN \to \IN$ be a computable function which tends to infinity with $\limsup_{n\to\infty} \sqrt[n]{u_f(n)} < 2$. Then the series $\sum_{k=0}^{\infty} 2^{-f(k)}$ converges and its limit is a reordered computable number.
\end{prop}

\begin{proof}
	The series $\sum_{k=0}^{\infty} u_f(k) \cdot 2^{-k}$ converges computably, due to Lemma \ref{lem:berechenbare-konvergenz-dyadische-reihe}. Therefore, the series $\sum_{k=0}^{\infty} 2^{-f(k)}$ converges. Since $f$ is computable, its limit is a left-computable number which is even reordered computable.
\end{proof}

\begin{defi}
	For every left-computable number $x$ we define
	\begin{equation*}
		\varsigma_x := \inf\left\{ \limsup_{n\to\infty} \sqrt[n]{u_f(n)} \bigm| f \textnormal{ is a computable name for } x\right\}
	\end{equation*}
	as the \emph{characteristic dyadic value} of $x$.
\end{defi}

The following properties of the characteristic dyadic value follow directly from our previous results.

\begin{prop}
	Let $x$ be a left-computable number. Then we have:
	\begin{enumerate}[(a)]
		\item $\varsigma_x \in \left[1, 2\right]$.
		\item\label{cdv<2=>roc} If $\varsigma_x \in \left[1, 2\right[$, then $x$ is reordered computable.
		\item\label{cdv-regular-1} If $x$ is regular, then $\varsigma_x = 1$.
		\item If $x$ is nearly computable, then $\varsigma_x = 1$ or $\varsigma_x = 2$.
		\item If $x$ is Martin-Löf random, then $\varsigma_x = 2$.
	\end{enumerate}
\end{prop}

\begin{proof} \;
	\begin{enumerate}[(a)]
		\item This follows from Proposition \ref{prop:lim-sup-schranken}.
		\item This follows from Proposition \ref{prop:limsup<2->roc}.
		\item This follows from Proposition \ref{prop:charakterisierung-regular}.
		\item If $x$ is not reordered computable, then $\varsigma_x = 2$, due to Statement~\ref{cdv<2=>roc}. Otherwise $x$ is reordered computable, and thus computable, according to Theorem \ref{satz:nc+roc=>computable}. Since computable numbers are regular, we have $\varsigma_x = 1$, due to Statement~\ref{cdv-regular-1}.
		\item This follows from Corollary \ref{kor:reordered-computable-implies-non-random}.
%
%
%
	\end{enumerate}
\end{proof}

We will later observe that there also exist a reordered computable number $x$ with $\varsigma_x = 2$ and a reordered computable number $y$ with $\varsigma_y = 1$ which is not regular. But first we will prove some other properties of the characteristic dyadic value.

\begin{satz}\label{satz:properties-of-varsigma}
	For left-computable numbers $x$ and $y$ the following hold:
	\begin{enumerate}[(a)]
		\item If $x \leq_S y$, then $\varsigma_x \leq \varsigma_y$.
		\item $\varsigma_{x+y} = \max\{\varsigma_x, \varsigma_y\}$
		\item $\varsigma_{x \cdot y} = \max\{\varsigma_x, \varsigma_y\}$
	\end{enumerate}
\end{satz}

\begin{proof}
	Denote $\rho := \max\{\varsigma_x, \varsigma_y\}$ for short and define the sequence $(\rho_k)_k$ with $\rho_k := \rho + 2^{-k}$ for every $k \in \IN$, which obviously converges to $\rho$.
	\begin{enumerate}[(a)]
		\item\label{statement-a} Let $(g_k)_k$ be a sequence of computable names for $y$ satisfying
		\begin{equation*}
			\limsup_{n\to\infty} \sqrt[n]{u_{g_k}(n)} \leq \varsigma_y + 2^{-k}
		\end{equation*}
		for all $k \in \IN$. Revisiting the proof of Theorem \ref{satz:closure-roc-solovay}, we observe that there exists a computable name $f_k \colon \IN\to\IN$ for $x$  and a constant $c \in \IN$ with $u_{f_k}(n) \leq n + \sum_{j=0}^{n+c} u_{g_k}(j)$ for all $n \geq 1$. We can conclude
		\begin{align*}
			\varsigma_x &\leq \limsup_{n\to\infty} \sqrt[n]{u_{f_k}(n)} \\
			&\leq \limsup_{n\to\infty} \sqrt[n]{n + \sum_{j=0}^{n+c} u_{g_k}(j)} \\
			&= \max\left\lbrace \limsup_{n\to\infty} \sqrt[n]{n}, \limsup_{n\to\infty} \sqrt[n]{\sum_{j=0}^{n+c} u_{g_k}(j)} \right\rbrace \\
			&= \max\left\lbrace 1, \limsup_{n\to\infty} \sqrt[n]{u_{g_k}(n)} \right\rbrace &\text{Lemma \ref{lem:limsup-g-summe-f}} \\
			&= \limsup_{n\to\infty} \sqrt[n]{u_{g_k}(n)} &\text{Proposition \ref{prop:lim-sup-schranken}} \\
			&\leq \varsigma_y + 2^{-k}
		\end{align*}
		for all $k \in \IN$, hence $\varsigma_x \leq \varsigma_y$.
		\item It is well-known that $\left[x + y\right]_S$ is the supremum of $\left[x\right]_S$ and $\left[y\right]_S$. So we have $\varsigma_{x + y} \geq \rho$, according to Statement \ref{statement-a}. We still have to show the other direction.
		Let $(f_k)_k$ and $(g_k)_k$ be sequences of computable names for $x$ and $y$, respectively, satisfying
		\begin{align*}
			\limsup_{n\to\infty} \sqrt[n]{u_{f_k}(n)} \leq \rho_k \\
			\limsup_{n\to\infty} \sqrt[n]{u_{g_k}(n)} \leq \rho_k
		\end{align*}
		for all $k \in \IN$. Revisiting the proof of Theorem \ref{prop:closure-sum}, we observe that there exists a computable name $h_k \colon \IN \to \IN$ for $x+y$ with $u_{h_k}(n) = u_{f_k}(n) + u_{g_k}(n)$ for all $n \in \IN$. We can conclude
		\begin{align*}
			\varsigma_{x + y} &\leq 
			\limsup_{n\to\infty} \sqrt[n]{u_{h_k}(n)} \\
			&= \limsup_{n\to\infty} \sqrt[n]{u_{f_k}(n) + u_{g_k}(n)} \\
			&= \max\left\lbrace \limsup_{n\to\infty} \sqrt[n]{u_{f_k}(n)}, \limsup_{n\to\infty} \sqrt[n]{u_{g_k}(n)}\right\rbrace  \\
			&\leq \rho_k
		\end{align*}
		for all $k \in \IN$, hence $\varsigma_{x+y} \leq \rho$.
		\item It is well-known that $\left[x \cdot y\right]_S$ is the supremum of $\left[x\right]_S$ and $\left[y\right]_S$. So we have $\varsigma_{x \cdot y} \geq \rho$, according to Statement \ref{statement-a}. We still have to show the other direction. Let $(f_k)_k$ and $(g_k)_k$ be sequences of computable names for $x$ and $y$, respectively, satisfying
		\begin{align*}
			\limsup_{n\to\infty} \sqrt[n]{u_{f_k}(n)} < \rho_k \\
			\limsup_{n\to\infty} \sqrt[n]{u_{g_k}(n)} < \rho_k
		\end{align*}
		for all $k \in \IN$. Then, there exists a constant $c_k > 0$ satisfying
		\begin{align*}
			u_{f_k}(n) \leq c_k \cdot \rho_k^n \\
			u_{g_k}(n) \leq c_k \cdot \rho_k^n 
		\end{align*}
		for all $n \in \IN$.
		Revisiting the proof of Theorem \ref{prop:closure-product}, we observe that there exists a computable name $h_k \colon \IN \to \IN$ for $x \cdot y$ with $u_{h_k}(n) = \sum_{j=0}^{n} u_{f_k}(j) \cdot u_{g_k}(n-j)$ for all $n \in \IN$. This implies
		\begin{align*}
			\varsigma_{x \cdot y} &\leq
			\limsup_{n\to\infty} \sqrt[n]{u_{h_k}(n)} \\ 
			&= \limsup_{n\to\infty} \sqrt[n]{\sum_{j=0}^{n} u_{f_k}(j) \cdot u_{g_k}(n-j)} \\
			&\leq \limsup_{n\to\infty} \sqrt[n]{\sum_{j=0}^{n} c_k \cdot \rho_k^j \cdot c_k \cdot \rho_k^{n-j}} \\
			&= \limsup_{n\to\infty} \sqrt[n]{c_k^2 \cdot \sum_{j=0}^{n} \rho_k^n} \\
			&= \limsup_{n\to\infty} \sqrt[n]{c_k^2 \cdot (n+1) \cdot \rho_k^n} \\
			&= \rho_k
		\end{align*}
		for all $k \in \IN$, hence $\varsigma_{x \cdot y} \leq \rho$.
	\end{enumerate}
\end{proof}

After this introduction we are ready for our last main result of this article.

\begin{satz}\label{satz:hierarchiesatz-roc}
	For every computably approximable number $\rho \in \left[ 1, 2\right]$ there exists a reordered computable number $x$ satisfying the following properties:
	\begin{enumerate}[(1)]
		\item\label{property-1} $\varsigma_x = \rho$
		\item\label{property-2} $x$ has computable name $f \colon \IN \to \IN$ with $\limsup_{n\to\infty} \sqrt[n]{u_f(n)} = \varsigma_x$.
		\item\label{property-3} $x$ is not regular.
	\end{enumerate}
\end{satz}

\begin{proof}
	Let $\rho \in \left[1, 2\right]$ be a computably approximable number. We will prove this theorem by constructing a computable non-decreasing sequence $(x_t)_t$ of rational numbers. At every stage $t\in\IN$, this sequence will either make a dyadic jump to the right or stay where it is. In order to ensure that $(x_t)_t$ converges to some reordered computable number, we will limit its jumps in advance.
	Fix some computable function $r \colon \IN \to \IN$ which tends to infinity and satisfies the two properties listed in Lemma~\ref{prop:berechenbare-konvergenz-berechenbar-approximierbar}.
	Then we can recursively compute an increasing function $s \colon \IN \to \IN$ satisfying
	\begin{itemize}
		\item $r(s(0)) \geq 1$,
		\item $r(s(n+1)) \geq 2 \cdot \sum_{k=0}^{n} r(s(k))$ and
		\item $\sum_{k=s(n+1)}^{\infty} r(k) \cdot 2^{-k} \leq \frac{1}{2} \cdot 2^{-s(n)}$
	\end{itemize}
	for all $n \in \IN$.
%
	During the construction of $(x_t)_t$ the algorithm will take care that this sequence makes at most $r(n)$ jumps of size $2^{-n}$. This will enforce that $(x_t)_t$ converges and that its limit, which is denoted by $x$, is a reordered computable number with $\varsigma_x \leq \rho$. On top of that, we define for every $e \in \IN$ the following infinite list of requirements:
	\begin{equation*}
		\mathcal{R}_e : \varphi_e \text{ total and a name for } x \implies \limsup_{n\to\infty} \sqrt[n]{u_{\varphi_e}(n)} \geq \rho
	\end{equation*}
	This enforces $\varsigma_x \geq \rho$ as well. In order to satisfy all these requirements, the algorithm performs an infinite diagonalization. Whenever some function $\varphi_{e}$ looks like a name of $x$, the requirement $\mathcal{R}_e$ will receive attention. We will react to this by performing some dyadic jump to the right with our sequence $(x_t)_t$. Since every left-computable number has infinitely many computable names, there will be infinitely many requirements receiving attention infinitely often. Therefore, $(x_t)_t$ also implicitly defines a computable name $f\colon\IN\to\IN$ for $x$.
	
	If such $\varphi_{e}$ actually is a name for $x$, the idea of this construction is that we want $\varphi_{e}$ to mimic the behavior of $f$. Note that we do not want to have for $\varphi_{e}$ that too many tiny jumps of $(x_t)_t$ are effectively replaced by some big jump, since this can lead to $\limsup_{n\to\infty} \sqrt[n]{u_{\varphi_{e}(n)}} < \rho$, which has to be prevented. Therefore, we have to choose the dyadic jumps of $(x_t)_t$ very carefully. During the entire construction, for all $n \in \IN$, only jumps of size $2^{-s(n)}$ will be used for $(x_t)_t$. This will ensure that the total sum of all minor jumps can never exceed the size of a single major jump. On top of that, for all $i, j \in \IN$, only jumps of size $2^{-s(\langle i, j \rangle)}$ will be used whenever $\mathcal{R}_i$ receives attention. Note that every big jump which is made during this construction can destroy the achievements of any prior smaller jumps. So there are conflicts between the requirements. However, this side-effect is already taken into account by the definition of $s$. We will later see that there are still enough successful jumps left, which are witnessed by $r$, to ensure $\varsigma_x \geq \rho$ and that $x$ is not regular.
	
	Furthermore, we recursively compute the two functions $w, c \colon \IN^2 \to \IN$. For all $e, t \in \IN$, the number $w(e)[t]$ basically encodes the current jump size which is used whenever $\mathcal{R}_e$ receives attention at stage $t$, and the number $c(e)[t]$ is a counter which tells us how many jumps of that size can still be made.
	
	After this motivation, we come to the formal description of this algorithm. At stage $0$ we define
	\begin{align*}
		x_0 &:= 0 \\
		w(e)[0] &:= 0 \\
		c(e)[0] &:= r(s(\left\langle e, 0\right\rangle)) - 1 
	\end{align*}
	for all $e \in \IN$. At stage $t+1$, compute the two numbers $i, j \in \IN$ with $\left\langle i, j\right\rangle = t$ and check whether the condition
	\begin{equation}\label{ineq:requires-attention}
		\left|x_t - \sum_{k\in W_i[t]}^{} 2^{-\varphi_i(k)}\right| \leq \frac{1}{8} \cdot 2^{-s(\left\langle i, w(i)[t]\right\rangle)}
	\end{equation}
	is met. If this is not the case, we set
	\begin{align*}
		x_{t+1} &:= x_t \\
		w(e)[t+1] &:= w(e)[t] \\
		c(e)[t+1] &:= c(e)[t]
	\end{align*}
	for all $e \in \IN$.
	Otherwise, we say that $\mathcal{R}_i$ \emph{receives attention} at this stage and we set
	\begin{align*}
		x_{t+1} &:= x_t + 2^{-s(\left\langle i, w(i)[t]\right\rangle)} \\
		w(e)[t+1] &:= \begin{cases}
			w(e)[t] + 1 &\text{if $e = i$ and $c(e)[t] = 0$} \\
			w(e)[t] &\text{otherwise}
		\end{cases} \\
		c(e)[t+1] &:= \begin{cases}
			r(s(\left\langle e, w(e)[t+1]\right\rangle )) - 1 &\text{if $e = i$ and $c(e)[t] = 0$} \\
			c(e)[t] - 1 &\text{if $e = i$ and $c(e)[t] > 0$} \\
			c(e)[t] &\text{otherwise}
		\end{cases}
	\end{align*}
	for all $e \in \IN$. This finishes the formal description of the algorithm.
	
	\begin{prop}
		$(x_t)_t$ is computable, non-decreasing and convergent. Hence, $x := \lim_{t \to \infty} x_t$ is a left-computable number.
	\end{prop}
	
	\begin{proof}
		Clearly, $(x_t)_t$ is computable and non-decreasing. We show that the series $\sum_{t=0}^{\infty} \left(x_{t+1} -x_t\right)$ converges. By construction, we have for all $i,j \in \IN$ at most $r(s(\left\langle i,j\right\rangle ))$ jumps of size $2^{-s(\left\langle i,j\right\rangle )}$. So we get
		\begin{align*}
			\sum_{t=0}^{\infty} \left(x_{t+1} - x_{t}\right)
			&\leq \sum_{i=0}^{\infty} \sum_{j=0}^{\infty} r(s(\left\langle i,j\right\rangle )) \cdot 2^{-s(\left\langle i,j\right\rangle )} \\
			&= \sum_{k=0}^{\infty} r(s(k)) \cdot 2^{-s(k)} \\
			&\leq \sum_{k=0}^{\infty} r(k) \cdot 2^{-k}.
		\end{align*}
	\end{proof}

	Obviously, there exist infinitely many $t \in \IN$ with $x_{t+1} > x_t$. Let $(t_n)_n$ be the uniquely determined computable increasing sequence of natural numbers which enumerates these stages, and let $f \colon \IN\to\IN$ be the uniquely determined computable function with $2^{-f(n)} = x_{t_n + 1} - x_{t_n}$ for all $n \in \IN$. So we get
	\begin{equation*}
		\sum_{k=0}^{\infty} 2^{-f(k)} 
		= \sum_{k=0}^{\infty} \left(x_{t_k + 1} - x_{t_k}\right) = \sum_{t=0}^{\infty} \left(x_{t+1} - x_{t}\right) = x.
	\end{equation*}
	Therefore, $f$ is a computable name for $x$.

	\begin{prop}
		The series $\sum_{k=0}^{\infty} u_f(k) \cdot 2^{-k}$ converges computably. Hence, $x$ is a reordered computable number.
	\end{prop}
	
	\begin{proof}
		By construction, we have $u_f(n) \leq r(n)$ for all $n \in \IN$. Therefore, the series $\sum_{k=0}^{\infty} u_f(k) \cdot 2^{-k}$ converges computably, and $x$ is a reordered computable number.
	\end{proof}
	
	In order to prove the next results, the next two lemmas are essential.
	\begin{lem}\label{lem:abschaetzung-spruenge-igel}
		Let $i \in \IN$, let $t_1 \in \IN$ be a stage at which $\mathcal{R}_i$ receives attention, and let $t_2 > t_1$ be the next stage at which $\mathcal{R}_i$ receives attention. 
		\begin{enumerate}[(a)]
			\item\label{lem:abschaetzung-spruenge-igel-1} Then we have the following lower bound:
			\begin{equation*}
				\sum_{k\in W_i[t_2]}^{} 2^{-\varphi_i(k)} - \sum_{k\in W_i[t_1]}^{} 2^{-\varphi_i(k)} \geq \frac{3}{4} \cdot 2^{-s(\left\langle i, w(i)[t_1] \right\rangle)}
			\end{equation*}
			\item\label{lem:abschaetzung-spruenge-igel-2} Furthermore, if for every $t \in \{t_1 + 1, \dots, t_2 - 1\}$ and for every $e \in \IN$ such that $\mathcal{R}_e$ receives attention at $t$ we have $\left\langle e, w(e)[t] \right\rangle > \left\langle i, w(i)[t_1] \right\rangle $, then we also have the following upper bound:
			\begin{equation*}
				\sum_{k\in W_i[t_2]}^{} 2^{-\varphi_i(k)} - \sum_{k\in W_i[t_1]}^{} 2^{-\varphi_i(k)} \leq \frac{7}{4} \cdot 2^{-s(\left\langle i, w(i)[t_1] \right\rangle)}
			\end{equation*}
		\end{enumerate}
	\end{lem}
	\begin{proof}
		Recall the condition that is satisfied (Inequality \ref{ineq:requires-attention}) when a requirement receives attention.
		\begin{enumerate}[(a)]
			\item Note that the following three inequalities hold:
			\begin{align*}
				x_{t_2} - x_{t_1} &\geq 2^{-s(\left\langle i, w(i[t_1])\right\rangle )} \\
				\sum_{k\in W_i[t_2]}^{} 2^{-\varphi_i(k)} - x_{t_2} &\geq - \frac{1}{8} \cdot 2^{-s(\left\langle i, w(i[t_1])\right\rangle )} \\
				x_{t_1} - \sum_{k\in W_i[t_1]}^{} 2^{-\varphi_i(k)} &\geq - \frac{1}{8} \cdot 2^{-s(\left\langle i, w(i[t_1])\right\rangle)} 
			\end{align*}
			The first inequality holds, since $\mathcal{R}_i$ receives attention at $t_1$ and $(x_t)_t$ is non-decreasing. The second inequality holds, since $\mathcal{R}_i$ receives attention at $t_2$ and we have $w(i)[t_2] \geq w(i)[t_1]$. The third inequality holds, $\mathcal{R}_i$ receives attention at $t_1$. Adding these three inequalities, we obtain the desired result.
			\item In this case the following three inequalities hold: \begin{align*}
				x_{t_2} - x_{t_1} &\leq \frac{3}{2} \cdot 2^{-s(\left\langle i, w(i[t_1])\right\rangle )} \\
				\sum_{k\in W_i[t_2]}^{} 2^{-\varphi_i(k)} - x_{t_2} &\leq  \frac{1}{8} \cdot 2^{-s(\left\langle i, w(i[t_1])\right\rangle )} \\
				x_{t_1} - \sum_{k\in W_i[t_1]}^{} 2^{-\varphi_i(k)} &\leq  \frac{1}{8} \cdot 2^{-s(\left\langle i, w(i[t_1])\right\rangle)} 
			\end{align*}
			By assumption, $\mathcal{R}_i$ receives attention at $t_1$ and the only numbers which are enumerated by $f$ between $t_1$ and $t_2$ are at least $s(\left\langle i, w(i)[t_1]\right\rangle+1 )$. Putting these two observations together, we obtain
			\begin{equation*}
				x_{t_2} - x_{t_1} \leq 2^{-s(\left\langle i, w(i[t_1])\right\rangle)} + \sum_{k=s(\left\langle i, w(i)[t_1]\right\rangle+1 )}^{\infty} r(k) \cdot 2^{-k} \leq \frac{3}{2} \cdot 2^{-s(\left\langle i, w(i[t_1])\right\rangle )}
			\end{equation*}
			by the definition of $s$. Therefore, the first inequality holds. The second inequality holds, since $\mathcal{R}_i$ receives attention at $t_2$ and we have $w(i)[t_2] \geq w(i)[t_1]$. The third inequality holds, since $\mathcal{R}_i$ receives attention at $t_1$. Adding these three inequalities, we obtain the desired result.			
		\end{enumerate}
	\end{proof}

	\begin{lem}\label{lem:master-lemma}
		Let $i \in \IN$ such that $\varphi_i$ is total and a name for $x$. Define the sequence $(m_j)_j$ by $m_j := s(\left\langle i,j\right\rangle)$ for all $j \in \IN$. Then the following statements hold:
		\begin{enumerate}[(a)]
			\item\label{lem:master-lemma-1} $\mathcal{R}_i$ receives attention at infinitely many stages.
			\item\label{lem:master-lemma-2} $u_f(m_j) = r(m_j)$ for all $j \in \IN$
			\item\label{lem:master-lemma-3} $\limsup_{n\to\infty} \sqrt[n]{u_f(n)} = \rho$
			\item\label{lem:master-lemma-4} $\sum_{k=0}^{\infty} u_{\varphi_i}(m_j+k) \cdot 2^{-(m_j+k)} \geq \frac{3}{8} \cdot  r(m_j) \cdot 2^{-m_j}$ for all $j \in \IN$
			\item\label{lem:master-lemma-5} $\limsup_{n\to\infty} \sqrt[n]{u_{\varphi_i}(n)} \geq \rho$
			\item\label{lem:master-lemma-6} $u_{\varphi_i}$ is not bounded.
		\end{enumerate}
	\end{lem}

	\begin{proof} \;
		\begin{enumerate}[(a)]
			\item This is clear.
			\item This follows from Statement \ref{lem:master-lemma-1} by construction.
			\item 
			Since we have $u_f(n) \leq r(n)$ for all $n \in \IN$, we obtain
			\begin{equation*}
				\limsup_{n\to\infty} \sqrt[n]{u_f(n)} \leq \limsup_{n\to\infty} \sqrt[n]{r(n)} = \lim_{n\to\infty} \sqrt[n]{r(n)} = \rho.
			\end{equation*}
			Using Statement \ref{lem:master-lemma-2}, we also obtain
			\begin{equation*}
				\limsup_{n\to\infty} \sqrt[n]{u_f(n)} 
				\geq \lim_{j\to\infty} \sqrt[n]{u_f(m_j)} 
				= \lim_{j\to\infty} \sqrt[n]{r(m_j)} 
				= \lim_{n\to\infty} \sqrt[n]{r(n)} 
				= \rho.
			\end{equation*}
			Therefore, we have $\limsup_{n\to\infty} \sqrt[n]{u_f(n)} = \rho$.
			\item Let $j \in \IN$. By the definition of $s$, there are more than $r(m_j)/2$ stages, denoted by $t_1, \dots, t_{r(m_j)/2}, t_{r(m_j)/2 + 1}$, satisfying the following:
			\begin{itemize}
				\item $\mathcal{R}_i$ receives attention at these stages.
				\item Between these stages  Statement~\ref{lem:abschaetzung-spruenge-igel-2} of Lemma \ref{lem:abschaetzung-spruenge-igel} is met.
				\item We have $w(i)[t] = j$ for all $t \in \{t_1, \dots, t_{r(m_j)/2} \}$.
			\end{itemize}
			Between these stages $\varphi_i$ has only enumerated numbers which are at least $m_j$. It is not possible that $\varphi_i$ has enumerated some number less than~$m_j$, due to the inequality of Statement~\ref{lem:abschaetzung-spruenge-igel-2} of Lemma \ref{lem:abschaetzung-spruenge-igel}. Since also Statement~\ref{lem:abschaetzung-spruenge-igel-1} of Lemma \ref{lem:abschaetzung-spruenge-igel} holds after these stages, we must have 
			\begin{equation*}
				\sum_{k=0}^{\infty} u_{\varphi_i}(m_j+k) \cdot 2^{-(m_j+k)} \geq \frac{1}{2} \cdot r(m_j) \cdot \frac{3}{4} \cdot 2^{-m_j} = \frac{3}{8} \cdot r(m_j) \cdot 2^{-m_j}.
			\end{equation*}
			\item Fix some function $h \colon \IN \to \IN$ with $\sum_{k=0}^{\infty} h(k) \cdot 2^{-k} < \frac{3}{8}$ such that the sequence $\left(\sqrt[n]{h(n)}\right)_n$ converges to $2$. In particular, we have
			\begin{equation}\label{ineq:reihe-mit-hilfsfunktion}
				\sum_{k=0}^{\infty} h(n+k) \cdot 2^{-(n+k)} < \frac{3}{8}
			\end{equation}
			for all $n \in \IN$.
			Combining Statement \ref{lem:master-lemma-4} with Inequality \ref{ineq:reihe-mit-hilfsfunktion}, this implies that there is a number $k_j \in \IN$ with
			\begin{equation*}\label{ineq:abschaetzung-von-u_f}
				u_{\varphi_i}(m_j + k_j) \geq h(m_j+k_j) \cdot r(m_j) \cdot 2^{-m_j}
			\end{equation*}
			and we obtain the following facts:
			\begin{align}
				\lim_{j \to \infty} \sqrt[m_j + k_j]{h(m_j+k_j)} &= 2 \label{fact-1} \\
				\limsup_{j \to \infty} \sqrt[m_j + k_j]{r(m_j) \cdot 2^{-m_j}} &\geq \frac{\rho}{2} \label{fact-2}
			\end{align}
			On the one hand, we have $\lim_{n\to\infty} \sqrt[n]{h(n)} = 2$. On the other hand, $(m_j)_j$ tends to infinity. Therefore, Fact~\ref{fact-1} is clear. In order to prove Fact~\ref{fact-2}, recall that we also have $\lim_{n\to\infty} \sqrt[n]{r(n) \cdot 2^{-n}} = \frac{\rho}{2}$. Due to $\rho \in \left[1,2\right]$, the actual value of the limit superior depends on the sequence $(k_j)_j$, but we can still conclude $\limsup_{j \to \infty} \sqrt[m_j + k_j]{r(m_j) \cdot 2^{-m_j}} \in \left[\frac{\rho}{2}, 1\right]$. Using Statement \ref{lem:master-lemma-4}, Fact \ref{fact-1} and Fact \ref{fact-2}, we obtain as desired
			\begin{align*}
				\limsup_{n\to\infty} \sqrt[n]{u_{\varphi_i}(n)} 
				&\geq \limsup_{j \to \infty} \sqrt[m_j + k_j]{h(m_j+k_j) \cdot r(m_j) \cdot 2^{-m_j}}  \\
				&= 2 \cdot \limsup_{j \to \infty} \sqrt[m_j + k_j]{r(m_j) \cdot 2^{-m_j}} \\
				&\geq 2 \cdot \frac{\rho}{2} \\
				&= \rho.
			\end{align*}
		\item Note that Statement \ref{lem:master-lemma-4} is equivalent to $\sum_{k=0}^{\infty} u_{\varphi_i}(m_j+k) \cdot 2^{-k} \geq \frac{3}{8} \cdot r(m_j)$ for all $j \in \IN$. This implies  \begin{equation}\label{ineq:einfache-abschaetzung-igel}
			u_{\varphi_i}(m_j+k) \geq \frac{3}{16} \cdot r(m_j)
		\end{equation}
		for some numer $k \in \IN$.
		
		Fix some arbitrary constant $c \in \IN$, and choose some number $j \in \IN$ with $r(m_j) > \frac{16}{3} \cdot c$. Note that such a number exists, since both $(m_j)_j$ and $r$ tend to infinity. Furthermore, choose some number $k \in \IN$ satisfying Inequality~\ref{ineq:einfache-abschaetzung-igel}. Then we obtain
		\begin{equation*}
			u_{\varphi_i}(m_j+k) \geq \frac{3}{16} \cdot r(m_j) > \frac{3}{16} \cdot \frac{16}{3} \cdot c = c.
		\end{equation*}
		Therefore, $u_{\varphi_i}$ is not bounded.
		\end{enumerate}
	\end{proof}
	
	Using Lemma \ref{lem:master-lemma}, we can prove the remaining properties of $x$. Statement~\ref{lem:master-lemma-3} implies $\varsigma_x \leq \rho$ and Statement \ref{lem:master-lemma-5} implies $\varsigma_x \geq \rho$. Therefore, we have $\varsigma_x = \rho$. Then, $f$ is a computable name for $x$ satisfying $\limsup_{n\to\infty} \sqrt[n]{u_f(n)} = \varsigma_x$. Finally, Statement \ref{lem:master-lemma-6} implies that $x$ is not regular. This ends the proof of this theorem.
\end{proof}

\begin{kor}
	There exists a reordered computable number $x$ with $\varsigma_x = 1$ which is not regular.
\end{kor}

\begin{kor}
	There exists a reordered computable number $x$ with $\varsigma_x = 2$.
\end{kor}

Note that for the reordered computable number $x$ constructed in Theorem~\ref{satz:hierarchiesatz-roc} there is a computable name $f \colon \IN \to \IN$ with $\limsup_{n\to\infty} \sqrt[n]{u_f(n)} = \varsigma_x$. It is unclear whether there exists a reordered computable number $y$ such that for every computable name $g \colon \IN \to \IN$ for $y$ we have $\limsup_{n\to\infty} \sqrt[n]{u_g(n)} > \varsigma_y$. We leave this as on open question.

\section{Acknowledgments}

The author would like to thank Peter Hertling for helpful discussions.

\bibliography{cca}
\bibliographystyle{abbrv}

\end{document}